\documentclass[reqno,11pt,centertags]{article}
\usepackage{amsmath,amsthm,amscd,amssymb,latexsym,upref}
\date{\today}

\usepackage[T2A,OT1]{fontenc}
\usepackage[russian,english]{babel}
\usepackage[margin=3.5 cm]{geometry}

\usepackage{hyperref}
\hypersetup{
    colorlinks=true, 
    linktoc=all,     
    linkcolor=blue,  
}

\newcommand{\bbD}{{\mathbb{D}}}

\newcommand{\bbT}{{\mathbb{T}}}

\newcommand{\z}{\zeta}

\allowdisplaybreaks \numberwithin{equation}{section}
\newtheorem{theorem}{Theorem}[section]

\newtheorem{lemma}[theorem]{Lemma}

\newtheorem{corollary}[theorem]{Corollary}

\theoremstyle{definition}

\newtheorem{remark}[theorem]{Remark}

\def\be{\begin{equation}}
\def\ee{\end{equation}}
\def\bea{\begin{eqnarray}}
\def\eea{\end{eqnarray}}
\def\bean{\begin{eqnarray*}}
\def\eean{\end{eqnarray*}}
\def\bbm{\begin{bmatrix}}
\def\ebm{\end{bmatrix}}

\def\restr#1{\,\vrule\,\lower1ex\hbox{$#1$}}

\def\a{\alpha}
\def\b{\beta}
\def\d{\delta}
\def\D{\Delta}

\def\g{\gamma}
\def\G{\Gamma}

\def\z{\zeta}

\def\C{{\bf C}}

\title
{Automorphic de Branges - Rovnyak Spaces}
\author{Alexander Kheifets}

\begin{document}

\maketitle

\begin{center}
\textit{Dedicated to the memory of Professor Heinz Langer}
\end{center}
\medskip


\begin{abstract}

A family of automorphic de Branges - Rovnyak
spaces (depending on an arbitrary character $\a$)
is associated to a $\b$-automorphic analytic function $w$
bounded in modulus by $1$. Multiplication by the Green function
defines unitary operators acting between those spaces. 
Function $w$ can be recovered from these operators as the characteristic function.
It is shown that for Nevanlinna-Pick problem this family of unitary operators extends
the family of isometric operators defined by the interpolation data.

\end{abstract}

\section{Automorphic Hardy Spaces}

\subsection{Green Function. Widom. Pommerenke.}

Let $\Gamma$ be a  Fuchsian group acting on the unit disk $\mathbb D$.
A mapping $\a$ from a group $\G$ to the unit circle $\bbT$ is called a character of $\G$ if
$$
\a(\g_1\g_2)=\a(\g_1)\a(\g_2).
$$
We say that function $f$ defined on the unit disk $\mathbb D$ or/and on the unit circle $\mathbb T$
is character automorphic (more specifically $\a$ automorphic) if
$$
f\circ\g=\a(\g) f,\quad\forall \g\in\G.
$$
Assuming that $\G$ is of convergent type, we define
the Green function of $\Gamma$ at $0$ as
$$
g_0(\zeta)=\prod\limits_{\gamma\in\Gamma}\gamma(\zeta)\frac{|\gamma(0)|}{\gamma(0)}.
$$
Group $\G$ is said to be of Widom type \cite{Widom71, Pom} if $g'_0$ is of bounded characteristic. In this case
\begin{theorem}[Widom-Pommerenke]\label{June19-02}
$$
g'_0(\zeta)=\dfrac{\Delta_0(\zeta)}{\psi_0(\zeta)},
$$
where $\Delta_0$ is an inner function 
\footnote{We will always use normalization $\D_0(0)>0$.}
and $\psi_0$ is a an outer function bounded in modulus
by~$1$.
\end{theorem}
Note that $g_0$ and $\D_0$ are character automorphic functions. We denote their characters as $\nu_0$
and $\d_0$, respectively.
\if{
We will also consider $g_{\zeta_0}(\zeta)$ the Green function of $\Gamma$ at $\zeta_0$
$$
g_{\zeta_0}(\zeta)=
\prod\limits_{\gamma\in\Gamma}\frac{\gamma(\zeta)-\zeta_0}{1-\gamma(\zeta)\overline{\zeta_0}}
c_{\gamma, \zeta_0}.
$$
For groups of Widom type we also have
$$
g'_{\z_0}(\zeta)=\dfrac{\Delta_{\z_0}(\zeta)}{\psi_{\z_0}(\zeta)},
$$
where $\Delta_{\z_0}$ is an inner function and $\psi_{\z_0}$ is a bounded outer function. We will use notations
$\nu_{\z_0}$ for the character of $g_{\z_0}$ and $\d_{\z_0}$ for the character of $\D_{\z_0}$.
}\fi

\subsection{Automorphic Hardy Spaces. Reproducing Kernels. Direct Cauchy Theorem.}\label{260111-01}

For groups of Widom type all spaces $H^p(\a)$, $1\le p\le\infty$, $\a\in\G^*$, contain nonconstant functions.
This can be seen by means of the Poincar\' e theta series
$$
(P^\alpha f) (t)=\dfrac{\sum\limits_{\gamma\in\Gamma}\overline{\alpha(\gamma)}f(\gamma(t))
|\gamma'(t)|}
{\sum\limits_{\gamma\in\Gamma}|\gamma'(t)|}
$$
$$
=\frac{\sum\limits_{\gamma\in\Gamma}\overline{\alpha(\gamma)}f(\gamma(t))
\dfrac{\gamma'(t)}{\g(t)}}
{\sum\limits_{\gamma\in\Gamma}\dfrac{\gamma'(t)}{\g(t)}}
=\frac{\sum\limits_{\gamma\in\Gamma}\overline{\alpha(\gamma)}f(\gamma(t))
\dfrac{\gamma'(t)}{\g(t)}}
{\dfrac{g'_0(t)}{g_0(t)}}.
$$
$P^\a$ is a contractive projection from $L^p$ onto $L^p(\a)$. For $p=2$, $P^\a$
is also selfadjoint (that is, it is the orthogonal projection from $L^2$ onto $L^2(\a)$). When applied to functions of
the form $\D_0 H^p$ it produces many functions in $H^p(\a)$. However, $P^\a$ is not the orthogonal projection from $H^2$
onto $H^2(\a)$.

By $k^\a_\z(t)$ we denote the orthogonal projection of $\dfrac{1}{1-t\overline\z}$ on $H^2(\a)$. It is the reproducing kernel of $H^2(\a)$ at point $\z$.

We say that the Direct Cauchy Theorem holds for a group of Widom type $\G$
(see \cite{Hasu, SY})
if for every $f\in H^1(\d_0)$ the following equality holds
\be\label{240901-01}
\int\limits_{\mathbb T}\frac{f}{\D_0}L(dt)=\frac{f(0)}{\D_0(0)},
\ee
where $L(dt)$ is the normalized Lebesgue measure on the unit circle $\bbT$. In particular, \eqref{240901-01} implies
that
$$
k^{\d_0}_0=\D_0\overline{\D_0(0)}
$$
is the reproducing kernel at $0$ for $H^2(\d_0)$.

The following orthogonal decomposition of
$L^2(\a)$ is an equivalent form of the Direct Cauchy Theorem property
(see \cite{Hasu, SY})
\be\label{250302-02}
H^2_\perp(\alpha)=L^2(\alpha)\ominus H^2(\alpha)
=\frac{\Delta_0}{g_0}\overline{H^2(\delta_0\overline{\nu_0}\ \overline\alpha)}.
\ee
Let $\widetilde k^\alpha_\zeta\in H^2_\perp(\alpha)$ be such that
$$
\left\langle
\dfrac{\D_0}{g_0}\overline{h}, \widetilde k^\alpha_\zeta
\right\rangle_{H^2_\perp(\alpha)}=\overline{h(\zeta)}.
$$
It is straightforward that
\be\label{June22-01}
\widetilde k^\alpha_\zeta
=\dfrac{\Delta_0}{g_0}\overline{k^{\delta_0 \overline{\nu_0}\ \overline\alpha}_\zeta}
.
\ee
The next lemma is a combination of Lemma 7.2 in \cite{SY}
and a weaker version of Lemma~3.2 in \cite{VY}.
We prove it here for the reader's convenience.
\begin{lemma}\label{250607_01}
Assume that the Direct Cauchy Theorem holds for group $\G$. Then
\be\label{250608_03}
\D_0\dfrac{\overline{k^\alpha_0}}{\sqrt{k^\alpha_0(0)}}
 =
\dfrac{k^{\d_0\overline\a}_0}{\sqrt{k^{\d_0\overline\a}_0(0)}}
\ee
and
\be\label{250608_04}
\D_0(0)
=\sqrt{k^\alpha_0(0)k^{\d_0\overline\a}_0 (0)}.
\ee
\end{lemma}
\begin{proof}
First we show that
$\D_0\overline{k^{\a}_0}\perp H^2_\perp(\d_0\overline\alpha)$,
that is,
$\D_0\overline{k^{\a}_0}\in H^2(\d_0\overline\alpha)$. Indeed,
by formula \eqref{250302-02},
$$
H^2_\perp(\d_0\overline\alpha)
=\frac{\Delta_0}{g_0}\overline{H^2(\overline{\nu_0}\ \alpha)}.
$$
Then for every $h\in {H^2(\overline{\nu_0}\ \alpha)}$
$$
\left\langle
\dfrac{\D_0}{g_0}\overline{h}, \D_0\overline{k^\alpha_0}
\right\rangle
=
\left\langle
k^\alpha_0 , g_0 h
\right\rangle
=0.
$$
Next, for every $u\in {H^2(\d_0\ \overline\alpha)}$
$$
\left\langle
u, \D_0\overline{k^\alpha_0}
\right\rangle
=
\int\limits_{\bbT}\dfrac{u(t)k^\alpha_0(t)}{\D_0(t)}L(dt)
=
\dfrac{u(0)k^\alpha_0(0)}{\D_0(0)},
$$
by the Direct Cauchy Theorem property. Therefore (since $\D_0(0)>0$),
\be\label{250608_01}
\D_0\overline{k^\alpha_0}
 =
\dfrac{k^\alpha_0(0)}{\D_0(0)}k^{\d_0\overline\a}_0.
\ee
Taking squares of the norms on the both sides of \eqref{250608_01} we get
$$
k^\alpha_0(0)
 =
\dfrac{k^\alpha_0(0)^2}{\D_0(0)^2}k^{\d_0\overline\a}_0 (0),
$$
that is,
\be\label{250608_02}
\D_0(0)
=
\sqrt{k^\alpha_0(0)k^{\d_0\overline\a}_0 (0)}.
\ee
Substituting \eqref{250608_02} into \eqref{250608_01} we get \eqref{250608_03}.
\end{proof}

\section{Automorphic de Branges-Rovnyak spaces}


Let $L^2$ be the space of square summable functions against the Lebesgue measure
 on the unit circle $\bbT$.
Let $w$ be an analytic, bounded in modulus by $1$ function on $\bbD$ and let
$W(t):=
\begin{bmatrix}
1& w(t) \\
\overline{w(t)} & 1 
\end{bmatrix}
$. 
The space $L^w$ is the range
space $W^{1/2}(L_2\oplus L_2)$ endowed with the range norm. In more
detail: for every element $f$ in $L^w$, there exists a unique $h_f\in
L_2\oplus L_2$ which is orthogonal to ${\rm Ker}\, W(t)$ for almost all
$t\in\mathbb T$ and such that $f=W^{1/2}h_f$. This unique $h_f$ will be denoted
by $h_f:=W^{[-1/2]}f$ and the $L^w$-norm of $f$ is defined as
$$
\|f\|_{L^w}^2:=\|h_{_f}\|_{L_2\oplus L_2}^2=
\int_{\mathbb T}\left\|
\begin{bmatrix}
1& w(t) \\
\overline{w(t)} & 1
\end{bmatrix}
^{[-1/2]}f(t)\right\|^2_{\C^2}L(dt),
$$
where $L(dt)$ stands for the normalized Lebesgue
measure on $\mathbb T$. The classical de Branges~-~Rovnyak space $H^w$
is defined (see \cite{BR}) as a subspace of $L^w$ that consists of functions
$$
f=\begin{bmatrix} f_2 \\ f_1\end{bmatrix},\quad f_2\in H^2_+,\quad f_1\in H^2_-,
$$
where $H^2_+$ and $H^2_-$ are the standard Hardy subspaces of $L^2$.

Assume that $w$ is character automorphic with character $\beta$. Note that then
$$
W\circ\gamma=
\begin{bmatrix}
1 & w\circ\gamma \\
\overline{w\circ\gamma} & 1
\end{bmatrix}
=
\begin{bmatrix}
\beta(\gamma) & 0 \\
0 & 1
\end{bmatrix}
\begin{bmatrix}
1 & w \\
\overline{w} & 1
\end{bmatrix}
\begin{bmatrix}
\overline{\beta(\gamma)} & 0
\\
0 & 1
\end{bmatrix}
=
\begin{bmatrix}
\beta(\gamma) & 0 \\
0 & 1
\end{bmatrix}
W
\begin{bmatrix}
\overline{\beta(\gamma)} & 0
\\
0 & 1
\end{bmatrix}
$$
for every $\g\in\G$.
The same relation holds for $W^{1/2}$. We consider now subspace $L^w(\alpha, \beta)$ of $L^w$
\be\label{250302-01}
\begin{bmatrix} f_2 \\ f_1 \end{bmatrix}\in L^w,\quad f_2\in L^2(\alpha\beta), \quad f_1\in L^2(\alpha),
\ee
where $L^2(\alpha)$ is a subspace of $L^2$ that consists of $\a$ automorphic functions.
To motivate: functions of the form
$$
\begin{bmatrix} f_2 \\ f_1 \end{bmatrix}=
W^{1/2}\begin{bmatrix} h_2 \\ h_1 \end{bmatrix},\quad h_2\in L^2(\alpha\beta), \quad h_1\in L^2(\alpha)
$$
satisfy \eqref{250302-01}. We define $H^w(\alpha, \beta)$ as a subspace of $L^w(\alpha, \beta)$ such that
$$
f_2\in H^2_+(\alpha\beta),\quad f_1\in H^2_\perp(\alpha)=
\dfrac{\D_0}{g_0}\overline{H^2(\overline{\nu_0}\d_0\overline\alpha)}
.
$$
Thus, to every $\b$-automorphic analytic function $w$, bounded in modulus by $1$,
one can associate a family of subspaces $L^w(\a,\b)$ and $H^w(\a,\b)$ that depend on
arbitrary character~$\a$.

\section{Associated Unitary Operators and Unitary Colligations}
Using \eqref{250302-02} one can get the following orthogonal decomposition
ot space $L^w(\a,\b)$
$$
L^w(\a,\b)=
\begin{bmatrix}
1 & w \\
\overline{w} & 1
\end{bmatrix}
\begin{bmatrix}
0
\\
H^2(\a)
\end{bmatrix}
\oplus
H^w(\a, \b)
\oplus
\begin{bmatrix}
1 & w \\
\overline{w} & 1
\end{bmatrix}
\begin{bmatrix}
H^2_{\perp}(\a\b)
\\
0
\end{bmatrix}
$$
$$
=
\begin{bmatrix}
1 & w \\
\overline{w} & 1
\end{bmatrix}
\begin{bmatrix}
0
\\
H^2(\a)
\end{bmatrix}
\oplus
H^w(\a, \b)
\oplus
\begin{bmatrix}
1 & w \\
\overline{w} & 1
\end{bmatrix}
\begin{bmatrix}
\dfrac{\D_0}{g_0}\overline{H^2(\overline{\nu_0}\d_0\overline{\alpha\b})}
\\
0
\end{bmatrix}.
$$
The operator of multiplication by $\overline{g_0}$ maps unitarily $L^w(\a,\b)$
onto $L^w(\overline{\nu_0}\a,\b)$.
Thus we have a family of unitary operators acting on the family of spaces $L^w(\a,\b)$.
Since any unitary operator preserves orthgonality, the
multiplication by $\overline{g_0}$  also maps unitarily
$$
\left\{
\dfrac{1}{\sqrt{k^\a_0(0)}}
\begin{bmatrix}
1 & w \\
\overline{w} & 1
\end{bmatrix}
\begin{bmatrix}
0
\\
k^\a_0
\end{bmatrix}
\right\}
\oplus
H^w(\a, \b)
$$
onto
$$
H^w(\overline{\nu_0}\a, \b)
\oplus
\left\{
\dfrac{1}{\sqrt{k^{\d_0\overline{\alpha\b}}_0(0)}}
\begin{bmatrix}
1 & w \\
\overline{w} & 1
\end{bmatrix}
\begin{bmatrix}
\dfrac{\D_0}{g_0}\overline{k^{\d_0\overline{\alpha\b}}_0}
\\
0
\end{bmatrix}
\right\}
$$
\be\label{241225-01}
=
H^w(\overline{\nu_0}\a, \b)
\oplus
\left\{
\dfrac{1}{\sqrt{k^{\alpha\b}_0(0)}}
\begin{bmatrix}
1 & w \\
\overline{w} & 1
\end{bmatrix}
\begin{bmatrix}
\overline{g_0}k^{\alpha\b}_0
\\
0
\end{bmatrix}
\right\}
.
\ee
The latter equality in \eqref{241225-01} 
is due to Lemma \ref{250607_01}.
Thus we get a family of unitary colligations. 
\begin{remark}
In the next section we show how to compute $w$ from the family of unitary colligations
\eqref{241225-01}, like the characteristic function. However, the author does not know what abstract object is modeled by
this family of colligations.
\end{remark}

\section{$w$ as the Characteristic Function.}

Since $\overline{g_0}$ maps 
$$
H^w(\overline{\nu_0^{n}}\a, \b)
\oplus
\left\{
\dfrac{1}{\sqrt{k^{\overline{\nu_0^{n}}\a}_0(0)}}
\begin{bmatrix}
1 & w \\
\overline{w} & 1
\end{bmatrix}
\begin{bmatrix}
0
\\
k^{\overline{\nu_0^{n}}\a}_0
\end{bmatrix}
\right\}
$$
unitarily onto
$$
H^w(\overline{\nu^{n+1}_0}\a, \b)
\oplus
\left\{
\dfrac{1}{\sqrt{k^{\overline{\nu^n_0}{\alpha\b}}_0(0)}}
\begin{bmatrix}
1 & w \\
\overline{w} & 1
\end{bmatrix}
\begin{bmatrix}
\overline{g_0} k^{\overline{\nu^n_0}{\alpha\b}}_0
\\
0
\end{bmatrix}
\right\}
,
$$
we consider this dynamics
\be\label{260225_03}
\overline{g_0}
\left(
h_n
\oplus
e_{1.n}
\right)
=
h_{n+1}
\oplus
e_{2,n},\quad n\ge 0,
\ee
where $h_n\in H^w(\overline{\nu_0^{n}}\a, \b)$,
\be\label{260225_04}
e_{1,n}:=
\dfrac{c_{1,n}}{\sqrt{k^{\overline{\nu_0^{n}}\a}_0(0)}}
\begin{bmatrix}
1 & w \\
\overline{w} & 1
\end{bmatrix}
\begin{bmatrix}
0
\\
k^{\overline{\nu_0^{n}}\a}_0
\end{bmatrix}
\ee
and
\be\label{260225_05}
e_{2,n}:=
\dfrac{c_{2,n}}{\sqrt{k^{\overline{\nu^n_0}{\alpha\b}}_0(0)}}
\begin{bmatrix}
1 & w \\
\overline{w} & 1
\end{bmatrix}
\begin{bmatrix}
\overline{g_0} k^{\overline{\nu^n_0}{\alpha\b}}_0
\\
0
\end{bmatrix}.
\ee
Note that $h_n$, $e_{1,n}$ and $e_{2,n}$ belong to different spaces for different $n$.
\begin{theorem}
Let $\a$ be an arbitrary character of $\G$.
Let $h_0=0$ and $c_{1,n}$ be an arbitrary $\ell^2$ input sequence.
Let $c_{2,n}$ be the corresponding output sequence from dynamics \eqref{260225_03}-\eqref{260225_05}.
Let
\be\label{260225_01}
\widetilde c_1=
\sum\limits_{n=0}^\infty
c_{1,n}
g^n_0\dfrac{k^{\overline{\nu_0^{n}}\a}_0}{\sqrt{k^{\overline{\nu_0^{n}}\a}_0(0)}}
\ee
and
\be\label{260225_02}
\widetilde c_2=
\sum\limits_{n=0}^\infty
c_{2,n}
g^n_0\dfrac{k^{\overline{\nu_0^{n}}\a\b}_0}{\sqrt{k^{\overline{\nu_0^{n}}\a\b}_0(0)}}
\ee
be the associated Fourier series $($functions in $H^2(\a)$ and in $H^2(\a\b)$, respectively$)$. Then
\be\label{260224_01}
\widetilde c_2=w \widetilde c_1.
\ee
\end{theorem}
\begin{proof}
We start with a special case when $c_{1,0}=1$ and $c_{1,n}=0$ for
$n\ge 1$. $h_0=0$, by assumption. Then for $n=0$ we have
$$
\overline{g_0}
\left(
0
\oplus
\dfrac{1}{\sqrt{k^{\a}_0(0)}}
\begin{bmatrix}
1 & w \\
\overline{w} & 1
\end{bmatrix}
\begin{bmatrix}
0
\\
k^{\a}_0
\end{bmatrix}
\right)
$$
$$
=
h_{1}
\oplus
\dfrac{c_{2,0}}{\sqrt{k^{{\alpha\b}}_0(0)}}
\begin{bmatrix}
1 & w \\
\overline{w} & 1
\end{bmatrix}
\begin{bmatrix}
\overline{g_0} k^{{\alpha\b}}_0
\\
0
\end{bmatrix}
,
$$
where
$$
c_{2,0}=
\left\langle
\dfrac{\overline{g_0}}{\sqrt{k^\a_0(0)}}
\begin{bmatrix}
1 & w \\
\overline{w} & 1
\end{bmatrix}
\begin{bmatrix}
0
\\
k^\a_0
\end{bmatrix}
,
\dfrac{1}{\sqrt{k^{{\alpha\b}}_0(0)}}
\begin{bmatrix}
1 & w \\
\overline{w} & 1
\end{bmatrix}
\begin{bmatrix}
\overline{g_0} k^{{\alpha\b}}_0
\\
0
\end{bmatrix}
\right\rangle
_{L^w(\overline{\nu_0}\a,\b)}
$$
$$
=
\left\langle
\dfrac{wk^\a_0}{\sqrt{k^\a_0(0)}}
,
\dfrac{k^{\alpha\b}_0}{\sqrt{k^{\alpha\b}_0(0)}}
\right\rangle
_{H^2(\a\b)}
=A^\a_0,
$$
where we will use this notations for the expansion of $\dfrac{wk^\a_0}{\sqrt{k^\a_0(0)}}$ with respect to the
standard orthonormal basis in $H^2(\a\b)$ (see \cite{SY})
\be\label{260117_01}
\dfrac{wk^\a_0}{\sqrt{k^\a_0(0)}}
=
\sum\limits_{n=0}^\infty
A^\a_n
g^n_0\dfrac{k^{\overline{\nu_0^{n}}\a\b}_0}{\sqrt{k^{\overline{\nu_0^{n}}\a\b}_0(0)}}.
\ee
Thus, we get
$$
e_{2,0}=
A^\a_0
\overline{g_0}
\begin{bmatrix}
1 & w \\
\overline{w} & 1
\end{bmatrix}
\begin{bmatrix}
\dfrac{k^{{\alpha\b}}_0}{\sqrt{k^{{\alpha\b}}_0(0)}}
\\ \\
0
\end{bmatrix}
$$
and
$$
h_1=
\overline{g_0}
\begin{bmatrix}
1 & w \\
\overline{w} & 1
\end{bmatrix}
\begin{bmatrix}
-A^\a_0\dfrac{k^{{\alpha\b}}_0}{\sqrt{k^{{\alpha\b}}_0(0)}}
\\ \\
\dfrac{k^\a_0}{\sqrt{k^\a_0(0)}}
\end{bmatrix}.
$$
Next, for $n=1$, we have
$$
\overline{g_0}
\left(
h_1
\oplus
0
\right)
$$
$$
=
h_{2}
\oplus
\dfrac{c_{2,1}}{\sqrt{k^{\overline{\nu_0}\alpha\b}_0(0)}}
\begin{bmatrix}
1 & w \\
\overline{w} & 1
\end{bmatrix}
\begin{bmatrix}
\overline{g_0} k^{\overline{\nu_0}\alpha\b}_0
\\
0
\end{bmatrix}
,
$$
where
$$
c_{2,1}=
\left\langle
\overline{g_0}h_1
,
\dfrac{1}{\sqrt{k^{\overline{\nu_0}\alpha\b}_0(0)}}
\begin{bmatrix}
1 & w \\
\overline{w} & 1
\end{bmatrix}
\begin{bmatrix}
\overline{g_0} k^{\overline{\nu_0}\alpha\b}_0
\\
0
\end{bmatrix}
\right\rangle
$$
$$
=
\left\langle
\overline{g_0}^2
\begin{bmatrix}
1 & w \\
\overline{w} & 1
\end{bmatrix}
\begin{bmatrix}
-A^\a_0\dfrac{k^{{\alpha\b}}_0}{\sqrt{k^{{\alpha\b}}_0(0)}}
\\ \\
\dfrac{k^\a_0}{\sqrt{k^\a_0(0)}}
\end{bmatrix}
,
\overline{g_0}
\begin{bmatrix}
1 & w \\
\overline{w} & 1
\end{bmatrix}
\begin{bmatrix}
\dfrac{k^{\overline{\nu_0}\alpha\b}_0}{\sqrt{k^{\overline{\nu_0}\alpha\b}_0(0)}}
\\ \\
0
\end{bmatrix}
\right\rangle
_{L^w(\overline{\nu_0}\a,\b)}
$$
$$
=
\left\langle
\left(
\dfrac{wk^\a_0}{\sqrt{k^\a_0(0)}}
-A^\a_0\dfrac{k^{{\alpha\b}}_0}{\sqrt{k^{{\alpha\b}}_0(0)}}
\right)
,
g_0
\dfrac{k^{\overline{\nu_0}\alpha\b}_0}{\sqrt{k^{\overline{\nu_0}\alpha\b}_0(0)}}
\right\rangle
_{H^2(\a\b)}
=A^\a_1.
$$
Thus, we get
$$
e_{2,1}=
A^\a_1
\overline{g_0}
\begin{bmatrix}
1 & w \\
\overline{w} & 1
\end{bmatrix}
\begin{bmatrix}
\dfrac{k^{\overline{\nu_0}\alpha\b}_0}{\sqrt{k^{\overline{\nu_0}\alpha\b}_0(0)}}
\\ \\
0
\end{bmatrix}
$$
and
$$
h_2=
\overline{g_0}^2
\begin{bmatrix}
1 & w \\
\overline{w} & 1
\end{bmatrix}
\begin{bmatrix}
-A^\a_0\dfrac{k^{{\alpha\b}}_0}{\sqrt{k^{{\alpha\b}}_0(0)}}
-A^\a_1 g_0\dfrac{k^{\overline{\nu_0}\alpha\b}_0}{\sqrt{k^{\overline{\nu_0}\alpha\b}_0(0)}}
\\ \\
\dfrac{k^\a_0}{\sqrt{k^\a_0(0)}}
\end{bmatrix}.
$$
Proceeding by induction, we get
$$
c_{2,n}=A^\a_n,
$$
$$
e_{2,n}=
A^\a_n
\overline{g_0}
\begin{bmatrix}
1 & w \\
\overline{w} & 1
\end{bmatrix}
\begin{bmatrix}
\dfrac{k^{\overline{\nu_0}^n\alpha\b}_0}{\sqrt{k^{\overline{\nu_0}^n\alpha\b}_0(0)}}
\\ \\
0
\end{bmatrix}
$$
and (for $n\ge 1$)
$$
h_n=
\overline{g_0}^n
\begin{bmatrix}
1 & w \\
\overline{w} & 1
\end{bmatrix}
\begin{bmatrix}
-A^\a_0\dfrac{k^{{\alpha\b}}_0}{\sqrt{k^{{\alpha\b}}_0(0)}}
-A^\a_1 g_0\dfrac{k^{\overline{\nu_0}\alpha\b}_0}{\sqrt{k^{\overline{\nu_0}\alpha\b}_0(0)}}
-\ldots
-A^\a_{n-1} g_0^{n-1}\dfrac{k^{\overline{\nu_0}^{n-1}\alpha\b}_0}{\sqrt{k^{\overline{\nu_0}^{n-1}\alpha\b}_0(0)}}
\\ \\
\dfrac{k^\a_0}{\sqrt{k^\a_0(0)}}
\end{bmatrix}.
$$
Therefore, the Fourier series \eqref{260225_01} and \eqref{260225_02} will be
$$
\widetilde c_1=
\sum\limits_{n=0}^\infty
c_{1,n}
g^n_0\dfrac{k^{\overline{\nu_0^{n}}\a}_0}{\sqrt{k^{\overline{\nu_0^{n}}\a}_0(0)}}
=
\dfrac{k^\a_0}{\sqrt{k^\a_0(0)}}.
$$
and
$$
\widetilde c_2=
\sum\limits_{n=0}^\infty
c_{2,n}
g^n_0\dfrac{k^{\overline{\nu_0^{n}}\a\b}_0}{\sqrt{k^{\overline{\nu_0^{n}}\a\b}_0(0)}}
=
\sum\limits_{n=0}^\infty
A^\a_n
g^n_0\dfrac{k^{\overline{\nu_0^{n}}\a\b}_0}{\sqrt{k^{\overline{\nu_0^{n}}\a\b}_0(0)}}
=
\dfrac{wk^\a_0}{\sqrt{k^\a_0(0)}}.
$$
Hence \eqref{260224_01} holds for the input sequence $c_{1,0}=1$, $c_{1,n}=0$ for $n\ge 1$.

Let now $c_{1,1}=1$ and $c_{1,n}=0$ for all other indices $n$. Then $h_1=0$, $c_{2,0}=0$ and
similar computation to the above starts with $c_{1,1}=1$ and $h_1=0\in H^w(\overline{\nu_0}\a, \b)$.
Therefore,
$$
c_{2,n}=A^{\overline{\nu_0}\a}_{n-1}, \quad n\ge 1 .
$$
Hence in this case
$$
\widetilde c_1=
\sum\limits_{n=0}^\infty
c_{1,n}
g^n_0\dfrac{k^{\overline{\nu_0^{n}}\a}_0}{\sqrt{k^{\overline{\nu_0^{n}}\a}_0(0)}}
=
g_0\dfrac{k^{\overline{\nu_0}\a}_0}{\sqrt{k^{\overline{\nu_0}\a}_0(0)}}.
$$
and
$$
\widetilde c_2
=
\sum\limits_{n=0}^\infty
c_{2,n}
g^n_0\dfrac{k^{\overline{\nu_0^{n}}\a\b}_0}{\sqrt{k^{\overline{\nu_0^{n}}\a\b}_0(0)}}
$$
$$
=
\sum\limits_{n=1}^\infty
A^{\overline{\nu_0}\a}_{n-1}
g^n_0\dfrac{k^{\overline{\nu_0^{n}}\a\b}_0}{\sqrt{k^{\overline{\nu_0^{n}}\a\b}_0(0)}}
=
\sum\limits_{n=0}^\infty
A^{\overline{\nu_0}\a}_{n}
g^{n+1}_0\dfrac{k^{\overline{\nu_0^{n+1}}\a\b}_0}{\sqrt{k^{\overline{\nu_0^{n+1}}\a\b}_0(0)}}
$$
$$
=
g_0\sum\limits_{n=0}^\infty
A^{\overline{\nu_0}\a}_{n}
g^{n}_0\dfrac{k^{\overline{\nu_0^{n+1}}\a\b}_0}{\sqrt{k^{\overline{\nu_0^{n+1}}\a\b}_0(0)}}
=
g_0
\dfrac{wk^{\overline{\nu_0}\a}_0}{\sqrt{k^{\overline{\nu_0}\a}_0(0)}}.
$$
Therefore, \eqref{260224_01} holds for this $c_1$ as well. Similar computation shows that
\eqref{260224_01} holds for every standard basic vector $c_1$ in $\ell^2$. Finally, due to the linearity of
dynamics \eqref{260225_03}-\eqref{260225_05},
\eqref{260224_01} holds for every $c_1\in\ell^2$.
\end{proof}

\section{Reproducing Kernels}
In what follows we assume that $w$ is automorphic with character $\beta$.
\begin{lemma}\label{June21-10}
The following functions
\be
K_{\zeta}^{\alpha, \beta}(t)=\left[\begin{array}{c} K_{\zeta,+}^{\alpha, \beta}(t) \\
K_{\zeta,-}^{\alpha, \beta}(t) \end{array}\right]
=\left[\begin{array}{cc} 1& w(t) \\
\overline{w(t)} & 1\end{array}\right]\left[\begin{array}{c} k^{\alpha\beta}_\zeta(t) \\
-\overline{w(\zeta)} k^\alpha_\zeta(t) \end{array}\right]
\label{12.8}
\ee
and
\be
\widetilde{K}_{\zeta}^{\alpha, \beta}(t)=
\left[\begin{array}{c} \widetilde{K}_{\zeta,+}^{\alpha, \beta}(t) \\
\widetilde{K}_{\zeta,-}^{\alpha, \beta}(t) \end{array}\right]
=\left[\begin{array}{cc} 1& w(t) \\
\overline{w(t)} & 1\end{array}\right]\left[\begin{array}{c} -w(\zeta)\widetilde k^{\alpha\beta}_\zeta(t) \\
\widetilde k^\alpha_\zeta(t) \end{array}\right]   \label{12.9}
\ee
are the reproducing kernels in $H^w(\alpha, \beta)$
for the first component at $\zeta$ and for the second component times
$\dfrac{g_0}{\D_0}$ at $\zeta$ $($which is a conjugate analytic function$)$, respectively.
Where $k^\alpha_\zeta$ and $\widetilde k^\alpha_\zeta$ are defined in Section \ref{260111-01}.
\end{lemma}
\begin{proof}
It follows immediately from definition \eqref{12.8} that
$$
K_{\zeta,+}^{\alpha, \beta}(t)=
k^{\alpha\beta}_\zeta(t)-w(t)\overline{w(\zeta)} k^\alpha_\zeta(t)
\in H^2(\alpha\beta).
$$
Clearly
$$
K_{\zeta,-}^{\alpha, \beta}(t)=
\overline{w(t)}k^{\alpha\beta}_\zeta(t)-\overline{w(\zeta)} k^\alpha_\zeta(t)\in L^2(\a).
$$
We show that it is orthogonal to $H^2(\a)$. For $h\in H^2(\a)$
$$
\left\langle h, K_{\zeta,-}^{\alpha, \beta}\right\rangle=
\left\langle h, \overline w k^{\alpha\beta}_\zeta-\overline{w(\zeta)} k^\alpha_\zeta\right\rangle
$$
$$
=
\left\langle w h, k^{\alpha\beta}_\zeta\rangle-w(\z)\langle h, k^\alpha_\zeta\right\rangle
=w(\z) h(\z) - w(\z) h(\z)=0.
$$
Thus, $K_{\zeta,-}^{\alpha, \beta}\in H^2_\perp(\a)$ and, therefore,
$K_{\zeta}^{\alpha, \beta}\in H^w(\alpha, \beta)$.
Now
$$
\left\langle\begin{bmatrix} f_2 \\ f_1 \end{bmatrix}, K_{\zeta}^{\alpha, \beta}\right\rangle_{H^w(\alpha, \beta)}
=
\left\langle\begin{bmatrix} f_2 \\ f_1 \end{bmatrix},
\begin{bmatrix} k^{\alpha\beta}_\zeta \\
-\overline{w(\zeta)} k^\alpha_\zeta \end{bmatrix}
\right\rangle_{L^2}
$$
$$
=
\left\langle f_2 , k^{\alpha\beta}_\zeta \right\rangle_{H^2(\alpha\beta)}
+
\left\langle f_1 , -\overline{w(\zeta)} k^{\alpha}_\zeta \right\rangle_{L^2 (\alpha)}
= f_2(\zeta),
$$
since $f_2\in H^2(\alpha\beta)$ and $f_1\in H^2_\perp(\alpha)$.

Proof of the second part is analogous. In view of \eqref{June22-01},
$$
\widetilde
K_{\zeta,-}^{\alpha, \beta}
=
\widetilde k^\alpha_\zeta(t)
-\overline{w(t)}w(\zeta)\widetilde k^{\alpha\beta}_\zeta(t)
\in
\frac{\D_0}{g_0}\overline{H^2(\overline\nu_0\d_0\overline\alpha)}
=
H^2_\perp(\alpha).
$$
Clearly
$$
\widetilde K_{\zeta,+}^{\alpha, \beta}
=
w(t)\widetilde k^\alpha_\zeta(t)
- w(\zeta)\widetilde k^{\alpha\beta}_\zeta(t)\in L^2(\a\b).
$$
We show that is is orthogonal to $H^2_\perp(\a\b)$.
Since
$$
H^2_\perp(\a\b)=\dfrac{\D_0}{g_0}\overline {H^2(\delta_0 \overline{\nu_0}\ \overline{\alpha\b})}
$$
we compute
$$
\left\langle\widetilde K_{\zeta,+}^{\alpha, \beta}, \dfrac{\D_0}{g_0}\overline h\right\rangle
=\left\langle
w \widetilde k^\alpha_\zeta
- w(\zeta)\widetilde k^{\alpha\beta}_\zeta, \dfrac{\D_0}{g_0}\overline h\right\rangle
$$
(using again \eqref{June22-01})
$$
=\left\langle
w \dfrac{\Delta_0}{g_0}\overline{k^{\delta_0 \overline{\nu_0}\ \overline\alpha}_\zeta}
- w(\zeta)\dfrac{\Delta_0}{g_0}\overline{k^{\delta_0 \overline{\nu_0}\ \overline{\alpha\b}}_\zeta}, \dfrac{\D_0}{g_0}\overline h\right\rangle
$$
$$
=\left\langle
w \overline{k^{\delta_0 \overline{\nu_0}\ \overline\alpha}_\zeta}
- w(\zeta)\overline{k^{\delta_0 \overline{\nu_0}\ \overline{\alpha\b}}_\zeta}, \overline h\right\rangle
$$
$$
=\left\langle
w h, {k^{\delta_0 \overline{\nu_0}\ \overline\alpha}_\zeta}
\right\rangle
-
w(\zeta)
\left\langle h, {k^{\delta_0 \overline{\nu_0}\ \overline{\alpha\b}}_\zeta}\right\rangle
$$
$$
=w(\z)h(\z)-w(\z)h(z)=0.
$$
Now
$$
\left\langle\begin{bmatrix} f_2 \\ f_1 \end{bmatrix}, \widetilde K_{\zeta}^{\alpha, \beta}\right\rangle_{H^w(\alpha, \beta)}
=
\left\langle\begin{bmatrix} f_2 \\ f_1 \end{bmatrix},
\begin{bmatrix} -w(\zeta)\widetilde k^{\alpha\beta}_\zeta \\
\widetilde k^\alpha_\zeta \end{bmatrix}
\right\rangle_{L^2}
$$
$$
=
\left\langle f_2 , -w(\zeta)\widetilde k^{\alpha\beta}_\zeta\right\rangle_{L^2(\alpha\beta)}
+
\left\langle f_1 , \widetilde k^\alpha_\zeta \right\rangle_{H^2_\perp (\alpha)}
=\left(\frac{g_0}{\D_0}f_1\right)(\zeta),
$$
since $f_2\in H^2(\alpha\beta)$ and $f_1\in H^2_\perp(\alpha)$.
\end{proof}

\begin{remark}\label{200524_01}
Since $k_\z^\a\perp H^2_\perp(\a)$, we have that
$$
\left[\begin{array}{cc} 1& w(t) \\
\overline{w(t)} & 1\end{array}\right]\left[\begin{array}{c} 0 \\
k^\alpha_\zeta(t) \end{array}\right]
\perp
H^w(\alpha, \beta).
$$
Then, in view of \eqref{12.8},
$$
K_{\zeta}^{\alpha, \beta}(t)
=\left[\begin{array}{cc} 1& w(t) \\
\overline{w(t)} & 1\end{array}\right]\left[\begin{array}{c} k^{\alpha\beta}_\zeta(t) \\
-\overline{w(\zeta)} k^\alpha_\zeta(t) \end{array}\right]
$$
$$
=P_{H^w(\a,\b)}
\left[\begin{array}{cc} 1& w(t) \\
\overline{w(t)} & 1\end{array}\right]\left[\begin{array}{c} k^{\alpha\beta}_\zeta(t) \\
-\overline{w(\zeta)} k^\alpha_\zeta(t) \end{array}\right]
$$
\be
=P_{H^w(\a,\b)}
\left[\begin{array}{cc} 1& w(t) \\
\overline{w(t)} & 1\end{array}\right]\left[\begin{array}{c} k^{\alpha\beta}_\zeta(t) \\
c k^\alpha_\zeta(t) \end{array}\right],
\label{20524_02}
\ee
where $c$ is any complex number, in particular $0$.
\end{remark}

\section{Towards Interpolation}
If $w$ satisfies certain interpolation conditions, we will have a family of isometric
colligations contained in the family of unitary colligations \eqref{241225-01}.
For example
\begin{lemma}
Multiplication by $\overline{g_0}$ maps
$$
\begin{bmatrix}
1 & w \\
\overline{w} & 1
\end{bmatrix}
\begin{bmatrix}
0
\\
k^\a_0
\end{bmatrix}
\dfrac{k^\a_\z(0)}{k^\a_0(0)}
\ \overline{w(\zeta)}
\oplus
K_{\zeta}^{\alpha, \beta}
$$
to
$$
K_{\zeta}^{\overline{\nu_0}\alpha, \beta}
\ \overline{g_0(\z)}
\oplus
\begin{bmatrix}
1 & w \\
\overline{w} & 1
\end{bmatrix}
\begin{bmatrix}
\dfrac{\D_0}{g_0}\overline{k^{\d_0\overline{\alpha\b}}_0}
\\
0
\end{bmatrix}
\dfrac{k^{\a\b}_\z(0)}{\D_0(0)}
.
$$
\end{lemma}
\begin{proof}
Consider
$$
A
\begin{bmatrix}
1 & w \\
\overline{w} & 1
\end{bmatrix}
\begin{bmatrix}
0
\\
k^\a_0
\end{bmatrix}
\oplus
\left[\begin{array}{cc} 1& w \\
\overline{w} & 1\end{array}\right]\left[\begin{array}{c} k^{\alpha\beta}_\zeta \\
-\overline{w(\zeta)} k^\alpha_\zeta \end{array}\right]
=
\left[\begin{array}{cc} 1& w \\
\overline{w} & 1\end{array}\right]\left[\begin{array}{c} k^{\alpha\beta}_\zeta \\
A k^\a_0-\overline{w(\zeta)} k^\alpha_\zeta \end{array}\right].
$$
We multiply this by $\overline{g_0}$ and compute the projection on
$\begin{bmatrix}
1 & w \\
\overline{w} & 1
\end{bmatrix}
\begin{bmatrix}
\dfrac{\D_0}{g_0}\overline{k^{\d_0\overline{\alpha\b}}_0}
\\
0
\end{bmatrix}.$
To this end we first compute
$$
\left\langle
\overline{g_0}
\left[\begin{array}{cc} 1& w \\
\overline{w} & 1\end{array}\right]\left[\begin{array}{c} k^{\alpha\beta}_\zeta \\
A k^\a_0-\overline{w(\zeta)} k^\alpha_\zeta \end{array}\right]
,
\begin{bmatrix}
1 & w \\
\overline{w} & 1
\end{bmatrix}
\begin{bmatrix}
\dfrac{\D_0}{g_0}\overline{k^{\d_0\overline{\alpha\b}}_0}
\\
0
\end{bmatrix}
\right\rangle_{L^w(\overline{\nu_0}\a.\b)}
$$
$$
=\langle
k^{\alpha\beta}_\zeta +
w(A k^\a_0-\overline{w(\zeta)} k^\alpha_\zeta)
,
\D_0 \overline{k^{\d_0\overline{\alpha\b}}_0}
\rangle_{L^2}
$$
using the Direct Cauchy Theorem
$$
=
\dfrac
{k^{\alpha\beta}_\zeta (0)+
w(0)(A k^\a_0(0)-\overline{w(\zeta)} k^\alpha_\zeta(0))}
{\D_0(0)}
k^{\d_0\overline{\alpha\b}}_0(0)
=
\dfrac
{k^{\alpha\beta}_\zeta (0)}
{\D_0(0)}
k^{\d_0\overline{\alpha\b}}_0(0),
$$
if $A=\overline{w(\zeta)} \dfrac{k^\a_\z(0)}{k^\a_0(0)}$.
Next we need to divide by
$$
\left\langle
\begin{bmatrix}
1 & w \\
\overline{w} & 1
\end{bmatrix}
\begin{bmatrix}
\dfrac{\D_0}{g_0}\overline{k^{\d_0\overline{\alpha\b}}_0}
\\
0
\end{bmatrix}
,
\begin{bmatrix}
1 & w \\
\overline{w} & 1
\end{bmatrix}
\begin{bmatrix}
\dfrac{\D_0}{g_0}\overline{k^{\d_0\overline{\alpha\b}}_0}
\\
0
\end{bmatrix}
\right\rangle_{L^w(\overline{\nu_0}\a.\b)}
=k^{\d_0\overline{\alpha\b}}_0(0).
$$
Thus, the requisite projection is
$$
\begin{bmatrix}
1 & w \\
\overline{w} & 1
\end{bmatrix}
\begin{bmatrix}
\dfrac{\D_0}{g_0}\overline{k^{\d_0\overline{\alpha\b}}_0}
\\
0
\end{bmatrix}
\dfrac{k^{\alpha\beta}_\zeta (0)}{\D_0(0)}
.
$$
Now we consider
$$
\overline{g_0}
\left[\begin{array}{cc} 1& w \\
\overline{w} & 1\end{array}\right]\left[\begin{array}{c} k^{\alpha\beta}_\zeta \\
\overline{w(\zeta)} \dfrac{k^\a_0}{k^\a_0(0)}k^\alpha_\zeta(0)-\overline{w(\zeta)} k^\alpha_\zeta \end{array}\right]
-
\dfrac
{k^{\alpha\beta}_\zeta (0)}
{\D_0(0)}
\begin{bmatrix}
1 & w \\
\overline{w} & 1
\end{bmatrix}
\begin{bmatrix}
\dfrac{\D_0}{g_0}\overline{k^{\d_0\overline{\alpha\b}}_0}
\\
0
\end{bmatrix}
$$
\be\label{240908-02}
=
\left[\begin{array}{cc} 1& w \\
\overline{w} & 1\end{array}\right]\left[\begin{array}{c}
k^{\alpha\beta}_\zeta\overline{g_0}
-
\dfrac{k^{\alpha\beta}_\zeta (0)}{\D_0(0)}
\dfrac{\D_0}{g_0}\overline{k^{\d_0\overline{\alpha\b}}_0}
 \\
\left(
\overline{w(\zeta)} \dfrac{k^\a_0}{k^\a_0(0)}k^\alpha_\zeta(0)-\overline{w(\zeta)} k^\alpha_\zeta
\right)
\overline{g_0}
-\overline w
\dfrac{k^{\alpha\beta}_\zeta (0)}{\D_0(0)}
\dfrac{\D_0}{g_0}\overline{k^{\d_0\overline{\alpha\b}}_0}
\end{array}\right]
\ee
Next we show that
\be\label{240908_01}
k^{\alpha\beta}_\zeta\overline{g_0}
-
\dfrac{k^{\alpha\beta}_\zeta (0)}{\D_0(0)}
\dfrac{\D_0}{g_0}\overline{k^{\d_0\overline{\alpha\b}}_0}
=\overline{g_0(\z)}k^{\overline{\nu_0}\alpha\beta}_\zeta .
\ee
Clearly this function is in $L^2(\overline{\nu_0}\alpha\beta)$. To show that it is in
$H^2(\overline{\nu_0}\alpha\beta)$ we prove that it is orthogonal to
$H^2_\perp(\overline{\nu_0}\alpha\beta)$. Indeed,
$$
\left\langle
k^{\alpha\beta}_\zeta\overline{g_0}
-
\dfrac{k^{\alpha\beta}_\zeta (0)}{\D_0(0)}
\dfrac{\D_0}{g_0}\overline{k^{\d_0\overline{\alpha\b}}_0}
,
\dfrac{\D_0}{g_0}\overline h
\right\rangle
$$
$$
=
\left\langle
k^{\alpha\beta}_\zeta h
,
\D_0
\right\rangle
-
\dfrac{k^{\alpha\beta}_\zeta (0)}{\D_0(0)}
\left\langle
\overline{k^{\d_0\overline{\alpha\b}}_0}
,
\overline h
\right\rangle
$$
by the Direct Cauchy Theorem and by the reproducing kernel property
$$
=
\dfrac
{k^{\alpha\beta}_\zeta(0) h(0)}
{\D_0(0)}
-
\dfrac{k^{\alpha\beta}_\zeta (0)}{\D_0(0)}
h(0)
=0.
$$
Now for $h\in H^2(\overline{\nu_0}\alpha\beta)$
$$
\left\langle
h,
k^{\alpha\beta}_\zeta\overline{g_0}
-
\dfrac{k^{\alpha\beta}_\zeta (0)}{\D_0(0)}
\dfrac{\D_0}{g_0}\overline{k^{\d_0\overline{\alpha\b}}_0}
\right\rangle
$$
$$
=
\left\langle
g_0 h,
k^{\alpha\beta}_\zeta
\right\rangle
-
\dfrac{k^{\alpha\beta}_\zeta (0)}{\D_0(0)}
\left\langle
h g_0 k^{\d_0\overline{\alpha\b}},
\D_0
\right\rangle
$$
$$
=
g_0(\z) h(\z)
-
\dfrac{k^{\alpha\beta}_\zeta (0)}{\D_0(0)}
\dfrac
{h(0) g_0(0) k^{\d_0\overline{\alpha\b}}(0)}
{\D_0(0)}
=
g_0(\z) h(\z)
=\left\langle h, \overline{g_0(\z)}k^{\overline{\nu_0}\alpha\beta}_\zeta\right\rangle .
$$
\eqref{240908_01} follows. Similarly one can verify that the second entry in \eqref{240908-02}
$$
\left(
\overline{w(\zeta)} \dfrac{k^\a_0}{k^\a_0(0)}k^\alpha_\zeta(0)-\overline{w(\zeta)} k^\alpha_\zeta
\right)
\overline{g_0}
-\overline w
\dfrac{k^{\alpha\beta}_\zeta (0)}{\D_0(0)}
\dfrac{\D_0}{g_0}\overline{k^{\d_0\overline{\alpha\b}}_0}
=
-\overline{g_0(\z)}\ \overline{w(\z)}k^{\overline{\nu_0}\alpha}_\zeta.
$$
This completes the proof of the lemma.
\end{proof}
\begin{corollary}\label{250112-01}
Multiplication by $\overline{g_0}$ maps $($isometrically$)$
$$
\sum\limits_k
\left(
\begin{bmatrix}
1 & w \\
\overline{w} & 1
\end{bmatrix}
\begin{bmatrix}
0
\\
k^\a_0
\end{bmatrix}
\dfrac{k^\a_{\z_k}(0)}{k^\a_0(0)}
\ \overline{w(\zeta_k)}
\oplus
K_{\zeta_k}^{\alpha, \beta}
\right)
x_k
$$
to
$$
\sum\limits_k
\left(
K_{\zeta_k}^{\overline{\nu_0}\alpha, \beta}
\ \overline{g_0(\z_k)}
\oplus
\begin{bmatrix}
1 & w \\
\overline{w} & 1
\end{bmatrix}
\begin{bmatrix}
\dfrac{\D_0}{g_0}\overline{k^{\d_0\overline{\alpha\b}}_0}
\\
0
\end{bmatrix}
\dfrac{k^{\a\b}_{\z_k}(0)}{\D_0(0)}
\right)x_k
.
$$
\end{corollary}
\begin{corollary}\label{250112-02}
By taking squares of the norms in Corollary \ref{250112-01}, we obtain the family of identities
$$
\sum\limits_{k,j}
\overline{x_j}
w(\z_j)
\dfrac{\overline{k^\a_{\z_j}(0)}k^\a_{\z_k}(0)}{k^\a_0(0)}
\ \overline{w(\zeta_k)}
x_k
+
\sum\limits_{k,j}
\overline{x_j}
\left(
k^{\alpha\beta}_{\z_k}(\zeta_j)-w(\zeta_j)\overline{w(\z_k)} k^\alpha_{\z_k}(\zeta_j)
\right)
x_k
$$
$$
=
\sum\limits_{k,j}
\overline{x_j}
\overline{\left(\dfrac{k^{\a\b}_{\z_j}(0)}{\D_0(0)}\right)}
k^{\d_0\overline{\alpha\b}}_0(0)
\dfrac{k^{\a\b}_{\z_k}(0)}{\D_0(0)}
x_k
$$
$$
+
\sum\limits_{k,j}
\overline{x_j}g_0(\z_j)
\left(
k^{\overline{\nu_0}\alpha\beta}_{\z_k}(\zeta_j)-w(\zeta_j)\overline{w(\z_k)} k^{\overline{\nu_0}\alpha}_{\z_k}(\zeta_j)
\right)
\overline{g_0(\z_k)}
x_k .
$$
\end{corollary}
\begin{corollary}
If a $\b$ automorphic function $w$, analytic on $\bbD$  and bounded in modulus by $1$, solves the Nevanlinna-Pick interpolation problem
\be\label{260227_01}
w(\z_k)=w_k,
\ee
then $w_k$ satisfy the family of identities (isometries)
$$
\sum\limits_{k,j}
\overline{x_j}
w_j
\dfrac{\overline{k^\a_{\z_j}(0)}k^\a_{\z_k}(0)}{k^\a_0(0)}
\ \overline{w_k}
x_k
+
\sum\limits_{k,j}
\overline{x_j}
\left(
k^{\alpha\beta}_{\z_k}(\zeta_j)-w_j\overline{w_k} k^\alpha_{\z_k}(\zeta_j)
\right)
x_k
$$
$$
=
\sum\limits_{k,j}
\overline{x_j}
\overline{\left(\dfrac{k^{\a\b}_{\z_j}(0)}{\D_0(0)}\right)}
k^{\d_0\overline{\alpha\b}}_0(0)
\dfrac{k^{\a\b}_{\z_k}(0)}{\D_0(0)}
x_k
$$
$$
+
\sum\limits_{k,j}
\overline{x_j}g_0(\z_j)
\left(
k^{\overline{\nu_0}\alpha\beta}_{\z_k}(\zeta_j)-w_j\overline{w_k} k^{\overline{\nu_0}\alpha}_{\z_k}(\zeta_j)
\right)
\overline{g_0(\z_k)}
x_k .
$$
and this family of isometries can be isometrically imbedded in the family of unitary colligations \eqref{241225-01}
via Corollary \ref{250112-01}.
\end{corollary}
\begin{remark}
We recall that positivity of all forms
$$
\sum\limits_{k,j}
\overline{x_j}
\left(
k^{\alpha\beta}_{\z_k}(\zeta_j)-w_j\overline{w_k} k^\alpha_{\z_k}(\zeta_j)
\right)
x_k \ge 0,
\quad \forall\a\in\G^*
$$
is necessary and sufficient for existence of a
$\b$ automorphic solution of problem \eqref{260227_01} (see \cite{Kup-Yud-1997}).
\end{remark}

\section{Declarations}
The author has no competing interests to declare that are relevant to the content of this article.
Data sharing is not applicable to this article as no new data were created or analyzed in this study.

\bigskip

\bigskip

A. Kheifets, Department of Mathematics and Statistics, University of Massachusetts Lowell, One University Ave.,
Lowell, MA 01854,USA

\emph{E-mail address:} {Alexander\underline{ }Kheifets@uml.edu}

 \end{document}